\documentclass[11pt]{article}

\usepackage[T1]{fontenc}
\usepackage[a4paper,margin=1in]{geometry}
\usepackage{amsmath,amssymb,amsthm,mathtools}
\usepackage{newtxtext,newtxmath}
\usepackage{needspace}
\usepackage{enumitem}
\usepackage{booktabs}
\usepackage{tikz}
\usepackage{float}
\usepackage{microtype}
\usepackage[nocompress]{cite}
\usepackage{hyperref}
\hypersetup{
  hidelinks,
  pdftitle={Nice Partitions of Signed-Graphic Arrangements Are Inductive Factorizations},
  pdfauthor={Weikang Liang, Suijie Wang, and Yue Zhou},
  pdfsubject={Hyperplane arrangements, signed graphs, and nice partitions},
  pdfkeywords={hyperplane arrangements, signed graphs, frame matroids, nice partitions,
    inductive factorizations, Terao conjecture}
}
\allowdisplaybreaks[3]

\theoremstyle{plain}
\newtheorem{theorem}{Theorem}[section]
\newtheorem{proposition}[theorem]{Proposition}
\newtheorem{corollary}[theorem]{Corollary}
\newtheorem{lemma}[theorem]{Lemma}
\theoremstyle{definition}
\newtheorem{definition}[theorem]{Definition}
\newtheorem{example}[theorem]{Example}
\theoremstyle{remark}

\newcommand{\A}{\mathcal A}

\newcommand{\cl}{\operatorname{cl}}
\newcommand{\rk}{\operatorname{rk}}
\numberwithin{equation}{section}

\renewenvironment{abstract}
  {\begin{list}{}{\setlength{\leftmargin}{\parindent}\setlength{\rightmargin}{\parindent}}%
   \item\relax\small\noindent\textsc{Abstract.}\ }
  {\end{list}}
\renewcommand{\footnoterule}{\kern-3pt\hrule width 1in\kern2.6pt}

\begin{document}
\begin{center}
  {\Large\bfseries
    Nice Partitions of Signed-Graphic Arrangements\\
    Are Inductive Factorizations\par}

  \vspace{18pt}

  {\normalsize
    Weikang Liang, Suijie Wang and Yue Zhou\par}
\end{center}

\begingroup
\renewcommand{\thefootnote}{}
\makeatletter
\let\@footnotetext\H@@footnotetext
\long\def\@makefntext#1{\parindent=1.8em\noindent\hspace*{\parindent}#1}
\makeatother
\footnotetext{%
\par
Corresponding author: Suijie Wang.}
\endgroup

\vspace{-6pt}

\begin{abstract}
Nice partitions of hyperplane arrangements give tensor factorizations
of Orlik--Solomon algebras as graded vector spaces. In general, a
hyperplane arrangement admitting a nice partition need not be free.
In this paper, we consider signed-graphic arrangements, the
subarrangements of type-\(B\) Weyl arrangements, over fields of
characteristic different from \(2\). We prove that every nice partition
of a signed-graphic arrangement is an inductive factorization and hence
that the arrangement is free, with nonzero exponents given by the part
sizes. In particular, Terao's conjecture holds for the intersection
lattices of signed-graphic arrangements admitting a nice partition.
\end{abstract}

\section{Introduction}

A central hyperplane arrangement is a finite set of linear hyperplanes
in a finite-dimensional vector space. In this paper, we study nice
partitions of signed-graphic arrangements over a field \(\mathbb K\)
of characteristic different from \(2\). These arrangements are the
subarrangements of the type-\(B\) Weyl arrangements and can be described
by signed graphs \cite{Zaslavsky1991,Zaslavsky2003}. A signed graph
\(\Sigma=(G^+,G^-,L)\) consists of two simple graphs on a common finite
vertex set \(V\), whose edges are positive and negative, respectively,
and a subset \(L\subseteq V\) specifying the vertices with loops.
The associated arrangement \(\A(\Sigma)\) in \(\mathbb K^V\) consists
of the hyperplanes \(x_i=x_j\) for positive edges \(ij\),
\(x_i=-x_j\) for negative edges \(ij\), and \(x_i=0\) for \(i\in L\).
Graphical arrangements are obtained when all edges are positive and
there are no loops.

A partition \(\pi\) of an arrangement \(\A\) is called
\emph{independent} if every choice of one hyperplane from each part
has linearly independent defining forms. It is called \emph{nice} if
it is independent and, for every intersection \(X\) of a nonempty
subfamily of \(\A\), some part contains exactly one hyperplane
containing \(X\). Every supersolvable arrangement admits a nice
partition. Terao~\cite{Terao1992} introduced nice partitions and proved
that a partition is nice if and only if it induces a tensor factorization
of the Orlik--Solomon algebra as a graded vector space.
In particular, the characteristic polynomial factors into linear
terms, with the sizes of the parts as its nonzero roots.
A similar factorization holds for free arrangements, with the
exponents as the roots of the characteristic polynomial by the
factorization theorem of Terao~\cite[Theorem~4.137]{OrlikTerao1992}.
Nevertheless, neither niceness nor freeness implies the other in
general \cite{Terao1992}. Jambu and Paris~\cite{JambuParis1995}
introduced \emph{inductive factorizations}, which are nice partitions
satisfying a recursive condition involving deletion and restriction,
and proved that every inductively factored arrangement is free.

M\"oller and R\"ohrle~\cite[Theorem~1.6]{MollerRoehrleNiceRestrictions2017}
proved that niceness and inductive factorization coincide for proper
restrictions of irreducible complex reflection arrangements.
For reflection arrangements themselves, however, Hoge and
R\"ohrle~\cite[Theorem~1.3 and Corollary~1.4]{HogeRoehrleNice2016}
showed that inductive factorization is equivalent to supersolvability
and is strictly stronger than niceness.
For signed-graphic arrangements, Edelman and Reiner~\cite{EdelmanReiner1994}
classified the free and the supersolvable real arrangements between
the type-\(A_{n-1}\) and type-\(B_n\) Weyl arrangements.
Bailey~\cite{BaileySignedGraphic} extended this study to nice partitions
and inductive factorizations, proving that the two properties coincide
in this interval and in several other families of signed-graphic
arrangements.
He conjectured the equivalence for all signed-graphic arrangements
\cite[Conjecture~4.5]{BaileySignedGraphic}.
In this paper, we prove Bailey's conjecture in the stronger form that
every nice partition of a signed-graphic arrangement is an inductive
factorization.

\begin{theorem}\label{thm:main-intro}
Let \(\Sigma=(G^+,G^-,L)\) be a finite signed graph, and let \(\pi\) be a
nice partition of \(\A(\Sigma)\). Then \(\pi\) is an inductive
factorization of \(\A(\Sigma)\).
\end{theorem}

The assertion concerns every nice partition, not merely the existence
of an inductive factorization. Indeed, an inductively factored
arrangement may admit nice partitions that are not inductive
factorizations \cite[Example~3.20]{HogeRoehrle2016}.
Since inductive factorization implies freeness, we also obtain the
following consequence.

\begin{corollary}\label{cor:nice-implies-free-intro}
Every signed-graphic arrangement that admits a nice partition is free.
\end{corollary}

Terao's conjecture~\cite[Conjecture~4.138]{OrlikTerao1992} asserts that
freeness over a fixed field is determined by the intersection lattice
\(L(\A)\), consisting of the intersections of subfamilies of
hyperplanes of \(\A\), ordered by reverse inclusion. Since niceness and inductive
factorization are determined by the intersection lattice, every
arrangement with the same intersection lattice as a nice signed-graphic
arrangement is inductively factored, and hence free. Terao's conjecture therefore holds for the intersection lattices of
signed-graphic arrangements admitting a nice partition. Moreover, if
\(\pi=(\pi_1,\ldots,\pi_r)\) is the original nice partition, every such
arrangement has nonzero exponents \(|\pi_1|,\ldots,|\pi_r|\).

For a real arrangement, the connected components of its complement
are called regions. Jambu and Paris~\cite{JambuParis1995} related
inductive factorizations to the poset of regions and its
rank-generating function. As a consequence of their results and our
theorem, every real signed-graphic arrangement with a nice partition
\(\pi=(\pi_1,\ldots,\pi_r)\) admits a base region for which the poset
of regions is a lattice with rank-generating function
\(\prod_{i=1}^r(1+t+\cdots+t^{|\pi_i|})\).

Section~\ref{sec:preliminaries} recalls the required facts about
signed-graphic arrangements, nice partitions, and inductive
factorizations. Section~\ref{sec:local-structure} describes the structure
of the parts and introduces anchored routes, which yield a safe element
when no coordinate hyperplane is present. Section~\ref{sec:coordinate-safe}
proves the existence of a safe element in general and completes the
proof of the main theorem. Section~\ref{sec:realization-space} develops
the consequences for freeness and Terao's conjecture, and concludes
with the consequence for regions of real arrangements.

\section{Preliminaries}
\label{sec:preliminaries}

\subsection{Signed-graphic arrangements}

Let \(\Sigma=(G^+,G^-,L)\) be a signed graph on a finite vertex set
\(V\), where \(G^+\) and \(G^-\) are simple graphs on \(V\) and
\(L\subseteq V\). The edges of \(G^+\) and \(G^-\) are the positive
and negative links, respectively. A loop is present at each vertex of
\(L\). For distinct vertices \(i,j\), write \(ij=\{i,j\}\).
There is a \emph{double edge} on \(ij\) if both links are present. The
\emph{underlying simple graph} has vertex set \(V\) and edge set
\(E(G^+)\cup E(G^-)\). For background on signed graphs and their
frame matroids, see Zaslavsky~\cite{Zaslavsky1991,Zaslavsky2003}.

Throughout, let \(\mathbb K\) be a field of characteristic different
from \(2\). In \(\mathbb K^V\), put \(h_i=\ker x_i\) and
\(H_{ij}^{\sigma}=\ker(x_i-\sigma x_j)\) for \(\sigma\in\{+1,-1\}\).
The signed-graphic arrangement associated with \(\Sigma\) is
\[
\A(\Sigma)=\{h_i:i\in L\}
\cup\{H_{ij}^{+}:ij\in E(G^+)\}
\cup\{H_{ij}^{-}:ij\in E(G^-)\}.
\]
Let \((e_i)_{i\in V}\) be the standard basis. The corresponding normal
vectors \(e_i\) and \(e_i-\sigma e_j\) represent the simple frame matroid
\(F(\Sigma)\). We identify each edge of \(\Sigma\) with the corresponding
hyperplane and matroid element, and call the elements \(h_i\) coordinate
elements. The support of an element is the set of indices at which its
defining normal is nonzero. Thus \(\operatorname{supp}(h_i)=\{i\}\) and
\(\operatorname{supp}(H_{ij}^{\sigma})=\{i,j\}\). The support of a set
of elements is the union of their supports. The opposite of a link
\(e=H_{ij}^{\sigma}\) is \(\bar e=H_{ij}^{-\sigma}\), when present.
A matroid is signed-graphic if it is isomorphic to the frame matroid
of a signed graph.

Cycles in this paper contain only links, and cycles of length two, called
digons, are allowed. A cycle is balanced if the product of its signs
is \(+1\), and unbalanced otherwise. A subgraph is balanced if it
contains no loop and every cycle is balanced. Switching by a function
\(\varepsilon:V\to\{+1,-1\}\) replaces the sign \(\sigma\) on
\(ij\) by \(\varepsilon_i\sigma\varepsilon_j\) and leaves the loops
unchanged. This operation is induced by the coordinate change
\(x_i\mapsto\varepsilon_i x_i\), so it preserves the frame matroid.

For an edge set \(S\), let \(V(S)\) be the set of vertices incident
with its edges, and let \(b(S)\) be the number of balanced components
of \((V(S),S)\). By \cite[Theorem~2.1(j)]{Zaslavsky1991},
\begin{equation}\label{eq:frame-rank}
	\rk_{F(\Sigma)}(S)=|V(S)|-b(S).
\end{equation}
A circuit is a minimally dependent set of matroid elements.
Every balanced cycle is a circuit and every unbalanced cycle is
independent, by the rank formula \eqref{eq:frame-rank}.
A three-element subset of \(\A(\Sigma)\) is a circuit if and only if
it has one of the following forms \cite{Zaslavsky1991}:
\begin{equation}\label{eq:signed-frame-three-circuits}
	\{h_i,h_j,H_{ij}^{\sigma}\},\qquad
	\{h_i,H_{ij}^{+},H_{ij}^{-}\},\qquad
	\{H_{ij}^{\alpha},H_{jk}^{\beta},H_{ik}^{\alpha\beta}\}.
\end{equation}
Here \(i\ne j\) in the first two forms. In the third, \(i,j,k\)
are distinct and the three links form a balanced triangle.

\subsection{Nice partitions}
\label{subsec:nice-partitions}

For a central arrangement \(\A\), write \(M(\A)\) for the simple matroid
represented by its defining forms. All matroid terminology for \(\A\)
refers to \(M(\A)\). A line is a rank-two flat. For a matroid \(M\),
write \(\cl_M\) for its closure operator and \(\chi_M(t)\) for its
characteristic polynomial. If \(\A\) has rank \(r\) in an
\(\ell\)-dimensional space, its characteristic polynomial satisfies
\(\chi_{\A}(t)=t^{\ell-r}\chi_{M(\A)}(t)\).

Given a simple matroid \(M\) on \(E\)
and a partition \(\pi=(\pi_1,\ldots,\pi_s)\) of \(E\), a partial
transversal of \(\pi\) is a set \(S\subseteq E\) containing
at most one element from each part, whereas a transversal contains exactly one
from each part. The partition \(\pi\) is \emph{independent} if every
transversal is independent in \(M\), or equivalently, if every partial
transversal is independent. For any flat \(X\), write
\(\pi_X=\{B\cap X:B\in\pi,\ B\cap X\ne\varnothing\}\) for the
partition induced on \(X\). The partition \(\pi\) is \emph{nice}
if it is independent and satisfies the \emph{singleton condition}
that \(\pi_X\) has a singleton part for every flat \(X\) of positive rank.
In a nice partition, no part contains a line.
For \(e\in E\), let \(\Pi(e)\) denote the part containing \(e\).
For \(S\subseteq E\), write
\(\Pi(S)=\{B\in\pi:B\cap S\ne\varnothing\}\).
We recall the following standard properties of nice partitions
\cite{OrlikTerao1992,Terao1992}.
\begin{proposition}\label{prop:nice-localizes}
 Let \(M\) be the simple matroid of a finite central arrangement and let
 \(\pi\) be a nice partition of \(M\). Write \(r=\rk M\). Then:
 \begin{enumerate}[label=\textup{(\roman*)}]
  \item \(\pi\) has exactly \(r\) parts;
  \item for every flat \(X\), \(\pi_X\) is a nice partition of \(M|X\);
  \item for every flat \(X\), \(|\pi_X|=\rk_M(X)\);
  \item the characteristic polynomial factors as
  \[
   \chi_M(t)=\prod_{B\in\pi}(t-|B|).
  \]
 \end{enumerate}
\end{proposition}

If \(N\) is a finite simple connected matroid represented by a central
arrangement and \(|E(N)|\ge2\), every nice partition of \(N\) has
exactly one singleton part. Indeed, Crapo's beta-invariant theorem
\cite{Crapo1967} implies \(\chi'_N(1)\ne0\). Since \(\chi_N(1)=0\),
the root \(1\) is simple. By
Proposition~\ref{prop:nice-localizes}\textup{(iv)}, its multiplicity
is the number of singleton parts.

If a simple matroid \(M=M_1\oplus M_2\) has a nice partition \(\pi\),
each part is contained in the ground set of one summand. Indeed, elements
from different summands form a two-point line, so
the singleton condition prevents them from lying in the same
part. Each summand is a restriction to a flat of \(M\), and the definition
of niceness shows that the induced partitions on both summands are nice.

The subarrangement corresponding to a flat of \(M(\A)\) is called a localization of \(\A\). In a signed-graphic arrangement \(\A(\Sigma)\), the elements supported
on a vertex set \(W\subseteq V\) form a flat. Their normals lie in
\(\mathbb K^W=\operatorname{span}\{e_w:w\in W\}\), whereas every
other element has a nonzero coordinate outside \(W\).
The corresponding subarrangement is therefore a localization. Any nice
partition of \(\A(\Sigma)\) induces a nice partition on it, by
Proposition~\ref{prop:nice-localizes}\textup{(ii)}.

\subsection{Inductive factorizations and safe elements}

Let \(\A\) be a finite central arrangement with a partition
\(\pi=(\pi_1,\ldots,\pi_s)\), and let \(H_0\in\pi_1\). Set
\(\A'=\A\setminus\{H_0\}\) and
\(\A''=\A^{H_0}=\{H_0\cap H:H\in\A'\}\), with repeated members
of \(\A''\) identified. Let \(\pi'\) be the partition of \(\A'\)
induced by \(\pi\), with empty parts omitted. The restriction map
\(\rho:\A\setminus\pi_1\longrightarrow\A''\) is given by
\(\rho(H)=H\cap H_0\). When \(\rho\) is bijective, the images
of \(\pi_2,\ldots,\pi_s\) form a partition \(\pi''\) of \(\A''\).
We define \emph{inductively factored pairs} recursively
\cite[Definition~3.8]{HogeRoehrle2016}:
\begin{enumerate}[label=\textup{(\roman*)}]
 \item The empty arrangement with its empty partition is inductively factored.
 \item For a nonempty arrangement \(\A\), the pair \((\A,\pi)\) is
 inductively factored if, after reordering the parts, there is an
 \(H_0\in\pi_1\) such that \(\rho\)
 is bijective and both \((\A',\pi')\) and \((\A'',\pi'')\) are
 inductively factored.
\end{enumerate}
If \((\A,\pi)\) is inductively factored, we call \(\pi\) an
\emph{inductive factorization} of \(\A\). An arrangement is
\emph{inductively factored} if it admits such a partition.
Every inductive factorization is a nice partition
\cite[Theorem~3.5]{HogeRoehrle2016}.
To construct a recursion from a nice partition, we use the following
result of M{\"o}ller--M{\"u}cksch--R{\"o}hrle
\cite[Proposition~3.4]{MollerMuckschRoehrle2024}.

\begin{proposition}
\label{prop:MMR-nice-children}
Let \(\pi\) be a nice partition of \(\A\), and let \(H_0\in\pi_1\).
If \(\rho:\A\setminus\pi_1\to\A''\) is bijective, then the induced
partitions \(\pi'\) and \(\pi''\) are nice.
\end{proposition}

To determine when the restriction map is bijective, we describe it in
the matroid \(M=M(\A)\). Write \(\operatorname{si}(N)\) for the
simplification of a matroid \(N\), obtained by deleting loops and
identifying parallel elements. If \(e\) corresponds to \(H_0\), then
\begin{equation}\label{eq:restriction-contraction}
 M(\A^{H_0})\cong\operatorname{si}(M/e).
\end{equation}
Restricting the defining forms to \(H_0\) amounts to taking their images
modulo the one-dimensional span of a defining form for \(H_0\).
Equal restricted hyperplanes correspond to parallel elements in \(M/e\).
Under this identification, \(\rho\) sends each element outside the part
\(B\) containing \(e\) to its parallel class in \(M/e\).
The same recursive definition therefore applies to finite simple
matroids, using deletion \(M\setminus e\) and simple contraction
\(\operatorname{si}(M/e)\). The deletion partition is obtained by removing
\(e\) from \(B\), omitting an empty part. When \(\rho\) is bijective,
the images of the parts other than \(B\) form the contraction partition.
Both niceness and inductive factorization are preserved by reordering
parts and by matroid isomorphisms, hence also by switching.

\begin{definition}\label{def:safe-element}
Let \(M\) be a simple matroid with a nice partition \(\pi\), and let
\(e\in E(M)\). A circuit \(\{e,x,y\}\) with
\(\Pi(x)=\Pi(y)\ne\Pi(e)\) is a \emph{bad three-point circuit through
\(e\)}. The element \(e\) is \emph{\(\pi\)-safe} if there is no such circuit.
\end{definition}

\begin{proposition}
\label{prop:three-circuit-deletion-bridge}
Let \(M=M(\A)\) be the simple matroid of a finite central arrangement,
let \(\pi\) be a nice partition of \(M\), and let \(e\in B\in\pi\).
The map \(\rho:E(M)\setminus B\longrightarrow E(\operatorname{si}(M/e))\)
sending each element to its parallel class in \(M/e\) is bijective
if and only if \(e\) is \(\pi\)-safe. In this case, the induced
deletion and contraction partitions are nice.
\end{proposition}

\begin{proof}
Let a parallel class of \(M/e\) be represented by \(f\ne e\).
If \(f\notin B\), this class is in the image of \(\rho\). Otherwise,
the line \(\cl_M\{e,f\}\) contains two elements of \(B\), so the
singleton condition gives an element \(g\) on this line outside \(B\).
Simplicity implies that \(\{e,f,g\}\) is a circuit, so \(f\) and \(g\)
are parallel in \(M/e\). Thus \(\rho\) is surjective. For distinct
\(x,y\notin B\), their images coincide exactly when \(\{e,x,y\}\)
is a circuit. Independence of \(\pi\) excludes this when \(x\) and \(y\)
lie in different parts. Hence \(\rho\) fails to be injective exactly
when there is a bad three-point circuit through \(e\). When \(\rho\)
is bijective, \eqref{eq:restriction-contraction} identifies it with the
arrangement restriction map, and Proposition~\ref{prop:MMR-nice-children}
gives the asserted niceness of both induced partitions.
\end{proof}

\section{Structure of parts and anchored routes}
\label{sec:local-structure}

Throughout this section, let \(\Sigma=(G^+,G^-,L)\) be a finite signed graph
on \(V\), write \(\A=\A(\Sigma)\), and fix a nice partition \(\pi\)
of \(\A\). The \emph{coordinate core} of a part \(B\in\pi\) is
\(K(B)=\bigcap_{H\in B}\operatorname{supp}(H)\), the set of vertices
common to the supports of all its elements.

\begin{lemma}
\label{lem:coordinate-core-cover}
If \(\rk\A=|V|\), then
\(\bigcup_{B\in\pi}K(B)=V\).
\end{lemma}
\begin{proof}
We have \(|\pi|=|V|\) by Proposition~\ref{prop:nice-localizes}\textup{(i)}.
Suppose that some \(v\in V\) lies in no core. For each \(B\in\pi\),
choose \(H_B\in B\) with \(v\notin\operatorname{supp}(H_B)\).
The set \(\{H_B:B\in\pi\}\) is a transversal, so its \(|V|\)
defining forms are linearly independent. But they all lie in the
\((|V|-1)\)-dimensional space of linear forms with zero coefficient
at \(x_v\), a contradiction.
\end{proof}

Adjoin an auxiliary vertex \(0\notin V\). The \emph{augmented support}
of \(h_i\) is \(\{0,i\}\), and that of \(H_{ij}^{\sigma}\) is
\(\{i,j\}\). A part is an \emph{augmented star} if its augmented
supports have a common vertex. Each such vertex is called a
\emph{centre} of the part. A \emph{\(B_2\)-part} is a part of the
form \(\{h_i,h_j,H_{ij}^{\sigma}\}\). A \emph{triangle part} is a
part consisting of the three links of an unbalanced triangle.

\begin{lemma}
\label{lem:global-support-trichotomy}
Every part of \(\pi\) is an
augmented star, a \(B_2\)-part, or a triangle part. For a \(B_2\)-part,
the opposite of its link is present outside that part. For a triangle part,
all three opposite links are present outside that part.
\end{lemma}

\begin{proof}
Let \(B\in\pi\). By the circuit list~\eqref{eq:signed-frame-three-circuits},
two elements with disjoint augmented supports cannot belong to a common
three-point circuit. They therefore form a two-point line, and the
singleton condition prevents them from both lying in \(B\).
Thus the distinct augmented supports in \(B\) form a pairwise
intersecting family of two-element sets. Such a family either has a
common vertex or consists of the three sides of a triangle.
In the first case, \(B\) is an augmented star.

Suppose that the augmented supports form a triangle containing \(0\),
say \(0ij\). Then \(B\) contains \(h_i,h_j,H_{ij}^{\sigma}\).
The only other element that can lie on their line is
\(H_{ij}^{-\sigma}\). The singleton condition forces this opposite link
to be present outside \(B\). Every element of \(B\) has augmented
support \(0i\), \(0j\), or \(ij\), so \(B\) consists of exactly the
three stated elements and is a \(B_2\)-part.

If the triangle does not contain \(0\), write it as \(ijk\) and choose
one link of \(B\) on each side. The line spanned by two chosen links
must contain an element outside \(B\), by the singleton condition.
By~\eqref{eq:signed-frame-three-circuits}, its only possible third
element is the link completing a balanced triangle. This link must
be opposite to the chosen link on the third side, so the chosen
triangle is unbalanced. Applying the same argument to the other two pairs gives
all three opposite links outside \(B\). Since the augmented supports
of \(B\) are \(ij,jk,ik\), it consists of exactly the three chosen
links and is a triangle part.
\end{proof}

\begin{definition}\label{def:anchored-route}
Let \(u\in V\setminus L\), and let \(q_0\) be a link incident with \(u\). An \emph{anchored route
at \(u\) starting from \(q_0\)} consists of \(q_0\) alone if \(q_0\)
is safe. Otherwise, it is a sequence of links constructed by the
following steps.
\begin{enumerate}[label=\textup{(\roman*)}]
 \item Given the current unsafe link \(q_i\), choose any bad three-point
 circuit through it. Since \(h_u\) is absent,
 \eqref{eq:signed-frame-three-circuits} allows us to write this circuit
 as \(\{q_i,q_{i+1},f_{i+1}\}\), where \(q_{i+1}\) is the unique
 other link incident with \(u\). Set
 \(B_{i+1}=\Pi(q_{i+1})=\Pi(f_{i+1})\).
 Since the circuit is bad, \(B_{i+1}\ne\Pi(q_i)\).
 \item Stop at \(q_{i+1}\) if it is safe or if \(B_{i+1}\) is a
 triangle part. Otherwise, repeat \textup{(i)} with \(q_{i+1}\) as the
 current link. We call \(B_1,B_2,\ldots\) the \emph{entered parts}
 of the route.
\end{enumerate}
\end{definition}

For an anchored route, write \(q_i=H_{uv_i}^{\varepsilon_i}\) and choose
the defining form \(\alpha_i=x_u-\varepsilon_i x_{v_i}\).
Each \(\alpha_i\) has coefficient \(1\) at \(x_u\), whereas every
defining form of \(f_{i+1}\) has coefficient \(0\) there.
In a linear dependence among these three forms, the coefficients of
\(\alpha_i\) and \(\alpha_{i+1}\) are therefore opposite. Hence
\(\beta_{i+1}:=\alpha_{i+1}-\alpha_i\) is a nonzero scalar multiple
of a defining form of \(f_{i+1}\).

\begin{samepage}
\begin{lemma}
\label{lem:anchored-route-no-return}
Let \(q_r,\ldots,q_s\), with \(r<s\), be a segment of an anchored route.
If the entered parts \(B_{r+1},\ldots,B_s\) are pairwise distinct, then:
\begin{enumerate}[label=\textup{(\roman*)}]
 \item \(q_s\ne q_r\);
 \item if \(H=\ker(\alpha_s-\alpha_r)\) belongs to \(\A\), then
 \(\Pi(H)\in\{B_{r+1},\ldots,B_s\}\).
\end{enumerate}
\end{lemma}
\end{samepage}

\begin{proof}
The elements \(f_{r+1},\ldots,f_s\) belong to distinct parts of the
independent partition \(\pi\), so their defining forms
\(\beta_{r+1},\ldots,\beta_s\) are linearly independent.
Summing the differences along the segment gives
\[
 \alpha_s-\alpha_r=\beta_{r+1}+\cdots+\beta_s.
\]
By linear independence, this nonempty sum is nonzero.
The distinct forms \(\alpha_r,\alpha_s\) both have coefficient \(1\)
at \(x_u\), so they are not proportional. Thus \(q_s\ne q_r\), proving
\textup{(i)}. If \(H=\ker(\alpha_s-\alpha_r)\) belongs to \(\A\)
and \(\Pi(H)\notin\{B_{r+1},\ldots,B_s\}\), the same identity gives
a nontrivial linear dependence among the defining forms of the partial
transversal \(\{f_{r+1},\ldots,f_s,H\}\), contrary to independence
of \(\pi\).
\end{proof}

\begin{lemma}
\label{lem:terminal-bad-route-hits-real-D3}
Every anchored route stops after finitely many steps, at a safe link or
in a triangle part. Its entered parts are pairwise distinct and are all
augmented stars, except for a possible final triangle part.
\end{lemma}

\begin{proof}
A \(B_2\)-part containing \(q_{i+1}\) would also contain \(h_u\),
which is absent. By Lemma~\ref{lem:global-support-trichotomy}, each
entered part \(B_{i+1}\) is therefore an augmented star or a
triangle part at which the route stops. If it is an augmented star, its centre
is \(v_{i+1}\), the unique common endpoint of the augmented supports
of \(q_{i+1}\) and \(f_{i+1}\).

Suppose that an augmented star is entered twice, and let \(B_s=B_r\),
with \(1\le r<s\), be the first repetition. Then the parts
\(B_{r+1},\ldots,B_s\) are pairwise distinct. The unique centre of
\(B_r=B_s\) gives \(v_s=v_r=v\). Thus \(q_s\) and \(q_r\) are either
the same link or the two opposite links on \(uv\).
Lemma~\ref{lem:anchored-route-no-return}\textup{(i)} excludes
\(q_s=q_r\).
The links \(q_r,q_s\) are therefore opposite. Since \(h_u\) is absent,
the only other element that can lie on their line is \(h_v\).
As \(q_r,q_s\in B_r\), the singleton condition forces \(h_v\) to be
present outside \(B_r\). We have
\(\alpha_s-\alpha_r=2\varepsilon_r x_v\ne0\), so
\(\ker(\alpha_s-\alpha_r)=h_v\). Hence
Lemma~\ref{lem:anchored-route-no-return}\textup{(ii)} gives
\(\Pi(h_v)=B_t\) for some \(r<t\le s\). Since \(h_v\notin B_s=B_r\),
we have \(t<s\). The part \(B_t\) is an augmented star containing
both \(h_v\) and a link \(q_t\) incident with \(u\), so its centre
must be \(v\). Thus \(q_t\) is one of the two links on \(uv\), both
already in \(B_r\), contrary to \(B_t\ne B_r\). No part is entered
twice, so the finiteness of \(\pi\) forces the route to stop at a
safe link or in a triangle part.
\end{proof}

\begin{lemma}
\label{lem:loopless-safe-element}
If \(L=\varnothing\) and \(\A\ne\varnothing\), then \(\A\) has a
\(\pi\)-safe element.
\end{lemma}

\begin{proof}
We may assume that \(M=F(\Sigma)\) is connected and that \(\Sigma\)
has no isolated vertices. Indeed, each matroid component is represented
by its signed subgraph and has a nice induced partition by
Subsection~\ref{subsec:nice-partitions}. Safety in a component implies
safety in \(M\), since every circuit lies in one component.
A one-element component gives a safe element, so assume \(|E(M)|\ge2\).
Suppose that no element is safe. For each \(u\in V\), start an
anchored route at \(u\) from an incident link. By
Lemma~\ref{lem:terminal-bad-route-hits-real-D3}, the route ends in a
triangle part whose support contains \(u\).

Let \(t\ge1\) be the number of triangle parts. The underlying simple graph
is connected, and a triangle part is unbalanced, so
\(\rk M=|V|\) by~\eqref{eq:frame-rank}. Hence
\(|\pi|=|V|\) by Proposition~\ref{prop:nice-localizes}\textup{(i)}.
By Subsection~\ref{subsec:nice-partitions}, there is exactly one
singleton part. Its coordinate core has size \(2\).
By Lemma~\ref{lem:global-support-trichotomy}, every other part is an
augmented star or a triangle part. Two opposite links form a
two-point line when \(L=\varnothing\), so they belong to distinct
parts. Consequently each nonsingleton augmented star has a core of
size \(1\), whereas each triangle part has empty core.
Lemma~\ref{lem:coordinate-core-cover} now gives
\[
 |V|\le\sum_{B\in\pi}|K(B)|=|\pi|+1-t=|V|+1-t.
\]
Thus there is a unique triangle part \(B\). Since every vertex
belongs to the support of a triangle part, \(V\) consists of the
three vertices supporting \(B\). Every element of \(B\) is safe.
Indeed, by
\eqref{eq:signed-frame-three-circuits}, a three-point circuit through
\(e\in B\) must be a balanced triangle on these three vertices.
The triangle formed by \(B\) is unbalanced. To obtain this balanced
triangle from \(B\) while keeping \(e\) fixed, exactly one of the
other two signs must be reversed. Thus the circuit contains another
element of \(B\) and is not bad, contrary to the assumption that
no element is safe.
\end{proof}

\section{Safe elements and the main theorem}
\label{sec:coordinate-safe}

Throughout this section, let \(\Sigma=(G^+,G^-,L)\) be a finite signed
graph on \(V\). Write \(\A=\A(\Sigma)\), \(M=F(\Sigma)\), and
\(U=V\setminus L\), and let \(\pi\) be a nice partition of \(M\).

\subsection{Local coordinate obstructions}
\label{subsec:local-coordinate-obstructions}

We first determine how connected components formed by double edges
on \(U\) can be joined to vertices of \(L\). To exclude the local
configurations that cannot admit a nice partition, we count pairs of
elements on lines. For a simple matroid \(N\), let \(\mathcal L_2(N)\)
be its set of lines and put
\[
 \lambda_2(N)=\sum_{X\in\mathcal L_2(N)}\binom{|X|-1}{2}.
\]

\begin{lemma}
\label{lem:line-excess-bound}
Let \(N\) be a simple matroid of rank \(r\ge2\) on \(m\) elements,
represented by a central arrangement. If \(N\) admits a nice partition, then
\[
 \lambda_2(N)\ge\frac{(m-1)(m-r)}{2(r-1)}.
\]
\end{lemma}

\begin{proof}
Choose a nice partition with part sizes \(d_1=1,d_2,\ldots,d_r\),
using the singleton condition and
Proposition~\ref{prop:nice-localizes}\textup{(i)}. The induced partition
on a line \(X\) has two parts, one of which is a singleton. Thus
\(\binom{|X|-1}{2}\) counts the pairs of elements of \(X\) that belong
to the same part. Every pair spans a unique line, so
\[
 \lambda_2(N)=\sum_{i=1}^r\binom{d_i}{2}.
\]
Since \(\sum_{i=2}^r d_i=m-1\), Cauchy's inequality gives
\[
 \lambda_2(N)\ge
 \frac12\left(\frac{(m-1)^2}{r-1}-(m-1)\right)
 =\frac{(m-1)(m-r)}{2(r-1)}.\qedhere
\]
\end{proof}

\begin{lemma}
\label{lem:one-hole-completion}
Suppose \(u\in U\) and \(x,y\in L\) are distinct. If
\(H_{ux}^{\alpha},H_{uy}^{\beta}\in\A\), then
\(H_{xy}^{\alpha\beta}\in\A\).
\end{lemma}

\begin{proof}
By switching at \(x\) and \(y\), we may assume \(\alpha=\beta=+1\).
Suppose that \(H_{xy}^{+}\) is absent. Let \(N\) be the matroid of
the localization on \(\{u,x,y\}\). It has rank three and a nice
induced partition by Proposition~\ref{prop:nice-localizes}\textup{(ii)}. Besides
\(h_x,h_y,H_{ux}^{+},H_{uy}^{+}\), the only possible elements are
\(H_{ux}^{-},H_{uy}^{-},H_{xy}^{-}\). Let \(a,b,c\in\{0,1\}\) record
their presence in this order, and put \(q=a+b+c\). Thus \(|E(N)|=4+q\).
The lines supported on the pairs \(ux,uy,xy\) contribute \(a,b,c\),
respectively, to \(\lambda_2(N)\). The two possible balanced triangles contribute
\(ac\) and \(bc\). These are all the lines with at least three points, so
\(\lambda_2(N)=a+b+c+ac+bc\). This contradicts
Lemma~\ref{lem:line-excess-bound}, since the difference between its
lower bound and \(\lambda_2(N)\) is
\(\bigl((q-2c)^2+3\bigr)/4>0\).
\end{proof}

\begin{lemma}
\label{lem:two-hole-local-defect}
Let \(u,v\in U\) be distinct and \(x\in L\). If both signs occur on
\(uv\) and a link joins \(u\) to \(x\), then both signs occur on
\(ux\) and \(vx\).
\end{lemma}

\begin{proof}
Suppose that one of the two links on \(vx\) is absent. By switching at \(u\) and
\(v\), we may assume that \(H_{ux}^{+}\) is present and \(H_{vx}^{+}\)
is absent. Let \(N\) be the matroid of the localization on
\(\{u,v,x\}\). It has rank three and inherits a nice partition.
Besides \(h_x,H_{uv}^{+},H_{uv}^{-},H_{ux}^{+}\), the only possible
elements are \(H_{ux}^{-}\) and \(H_{vx}^{-}\). Let \(b,c\in\{0,1\}\)
record their presence, respectively. Then \(N\) has \(4+b+c\) elements.
The line supported on the pair \(ux\) contributes \(b\) to
\(\lambda_2(N)\), and the two possible balanced triangles contribute
\(bc\) and \(c\). These are all the lines with at least three points,
so \(\lambda_2(N)=b+bc+c\). This contradicts
Lemma~\ref{lem:line-excess-bound}, since the difference between its
lower bound and \(\lambda_2(N)\) is
\(\bigl((b-c)^2+3\bigr)/4>0\). Thus both signs occur
on \(vx\). Applying the same argument with \(u,v\) interchanged gives
both signs on \(ux\).
\end{proof}

Let \(D_U\) be the graph on \(U\) whose edges are the pairs carrying both
signs. For the vertex set \(C\) of a component of \(D_U\), let
\(S_C\subseteq L\) consist of the vertices joined by a link to some
vertex of \(C\).

\begin{lemma}
\label{lem:double-component-propagation}
Let \(C\) be the vertex set of a component of \(D_U\) with \(|C|\ge2\).
Both signs occur between every \(c\in C\) and \(x\in S_C\), and between
every two distinct vertices of \(S_C\).
\end{lemma}

\begin{proof}
Fix \(x\in S_C\). Some vertex of \(C\) is joined to \(x\), and
Lemma~\ref{lem:two-hole-local-defect}, applied along paths in the
connected graph \(D_U[C]\), gives both signs on every \(cx\) with
\(c\in C\). For distinct \(x,y\in S_C\), choose \(c\in C\).
Both signs occur on \(cx\) and \(cy\), so
Lemma~\ref{lem:one-hole-completion} gives both signs on \(xy\).
\end{proof}

For positive integers \(s,\ell\), let \(\mathsf F_{s,\ell}\) be the
arrangement defined by the signed graph with vertices
\(x_1,\ldots,x_s,c,z_1,\ldots,z_\ell\) and loop set
\(S=\{x_1,\ldots,x_s\}\). Every pair of distinct vertices carries
both signs except the pairs \(z_i z_j\), which carry only a positive link.
Figure~\ref{fig:F-models} shows \(\mathsf F_{1,2}\), \(\mathsf F_{2,1}\),
and \(\mathsf F_{2,2}\). Let \(N_{\triangle}\) and \(\mathsf P_j\),
for \(j=0,1,2\), be the arrangements defined by the signed graphs in
Figure~\ref{fig:finite-models}. In both figures, \(\pm\) denotes both signed
links, \(+\) denotes a single positive link, and a loop at \(v\)
represents \(h_v\).

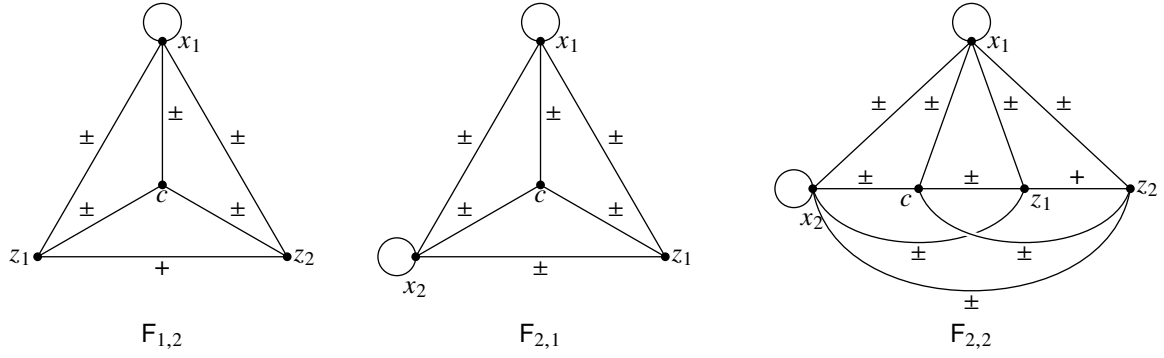
\begin{figure}[H]
\centering
\begin{tikzpicture}[
  x=1cm,y=1cm,
  every node/.style={font=\small,inner sep=2pt},
  edge/.style={line width=.5pt},
  sign/.style={inner sep=1.5pt},
  crossing/.style={preaction={draw=white,line width=2.5pt}}
]
\begin{scope}
  \coordinate (x1) at (0,1.95);
  \coordinate (z1) at (-1.65,-.9);
  \coordinate (z2) at (1.65,-.9);
  \coordinate (c) at (0,.05);
  \foreach \p/\q/\side in {x1/z1/above left,x1/z2/above right,x1/c/right,c/z1/above left,c/z2/above right}
    \draw[edge] (\p) -- node[sign,\side] {\(\pm\)} (\q);
  \draw[edge] (z1) -- node[sign,below] {\(+\)} (z2);
  \draw[edge] (0,2.2) circle[radius=.25cm];
  \foreach \p in {x1,z1,z2,c} \fill (\p) circle[radius=1.65pt];
  \node[right,inner sep=6pt,yshift=-1pt] at (x1) {\(x_1\)};
  \node[left] at (z1) {\(z_1\)};
  \node[right] at (z2) {\(z_2\)};
  \node[below] at (c) {\(c\)};
  \node at (0,-1.95) {\(\mathsf F_{1,2}\)};
\end{scope}
\begin{scope}[xshift=5cm]
  \coordinate (x1) at (0,1.95);
  \coordinate (x2) at (-1.65,-.9);
  \coordinate (z1) at (1.65,-.9);
  \coordinate (c) at (0,.05);
  \foreach \p/\q/\side in {x1/x2/above left,x1/z1/above right,x1/c/right,x2/z1/below,x2/c/above left,z1/c/above right}
    \draw[edge] (\p) -- node[sign,\side] {\(\pm\)} (\q);
  \draw[edge] (0,2.2) circle[radius=.25cm];
  \draw[edge] (-1.9,-.9) circle[radius=.25cm];
  \foreach \p in {x1,x2,z1,c} \fill (\p) circle[radius=1.65pt];
  \node[right,inner sep=6pt,yshift=-1pt] at (x1) {\(x_1\)};
  \node[below,inner sep=9pt] at (x2) {\(x_2\)};
  \node[right] at (z1) {\(z_1\)};
  \node[below] at (c) {\(c\)};
  \node at (0,-1.95) {\(\mathsf F_{2,1}\)};
\end{scope}
\begin{scope}[xshift=10.7cm]
  \coordinate (x1) at (0,1.95);
  \coordinate (x2) at (-2.1,0);
  \coordinate (z1) at (.7,0);
  \coordinate (c) at (-.7,0);
  \coordinate (z2) at (2.1,0);
  \foreach \p/\q/\side in {x1/x2/above left,x1/c/above left,x1/z1/above right,x1/z2/above right,x2/c/above,c/z1/above}
    \draw[edge] (\p) -- node[sign,\side] {\(\pm\)} (\q);
  \draw[edge] (z1) -- node[sign,above] {\(+\)} (z2);
  \draw[edge] (x2) .. controls (-1.8,-.95) and (.4,-.95) ..
    node[sign,below,pos=.5] {\(\pm\)} (z1);
  \draw[edge,crossing] (c) .. controls (-.4,-.95) and (1.8,-.95) ..
    node[sign,below,pos=.5] {\(\pm\)} (z2);
  \draw[edge] (x2) .. controls (-1.8,-1.8) and (1.8,-1.8) ..
    node[sign,below,pos=.5] {\(\pm\)} (z2);
  \draw[edge] (0,2.2) circle[radius=.25cm];
  \draw[edge] (-2.35,0) circle[radius=.25cm];
  \foreach \p in {x1,x2,z1,c,z2} \fill (\p) circle[radius=1.65pt];
  \node[right,inner sep=6pt,yshift=-1pt] at (x1) {\(x_1\)};
  \node[below,inner sep=9pt] at (x2) {\(x_2\)};
  \node[below right] at (z1) {\(z_1\)};
  \node[below left] at (c) {\(c\)};
  \node[right] at (z2) {\(z_2\)};
  \node at (0,-1.95) {\(\mathsf F_{2,2}\)};
\end{scope}
\end{tikzpicture}
\caption{The signed graphs of \(\mathsf F_{1,2}\), \(\mathsf F_{2,1}\),
and \(\mathsf F_{2,2}\).}
\label{fig:F-models}
\end{figure}

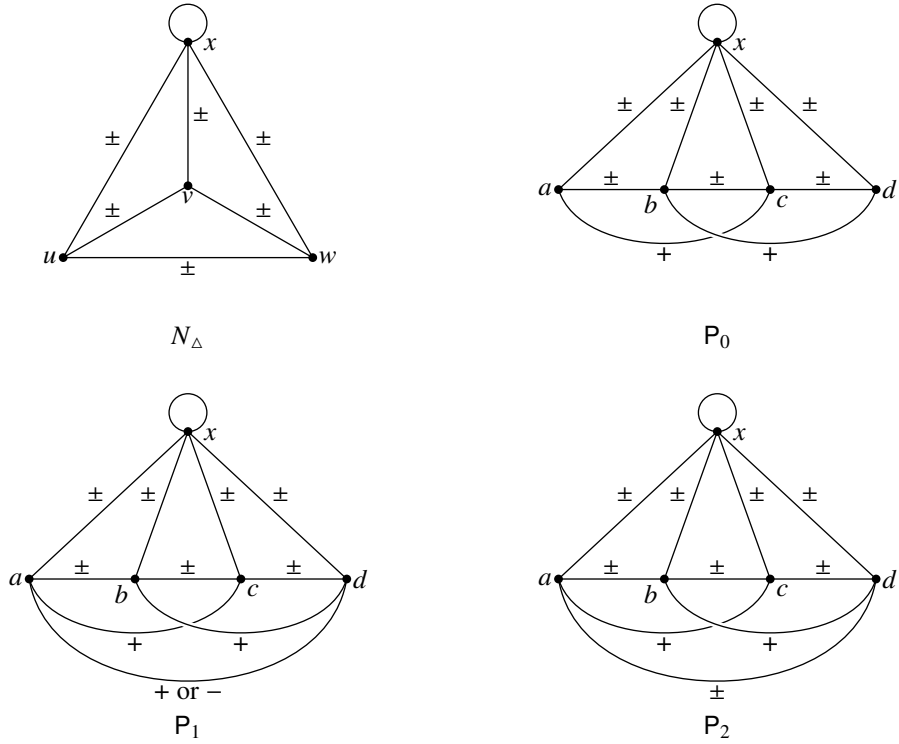
\begin{figure}[H]
\centering
\begin{tikzpicture}[
  x=1cm,y=1cm,
  every node/.style={font=\small,inner sep=2pt},
  edge/.style={line width=.5pt},
  sign/.style={inner sep=1.5pt},
  crossing/.style={preaction={draw=white,line width=2.5pt}}
]
\begin{scope}
  \coordinate (x) at (0,1.95);
  \coordinate (u) at (-1.65,-.9);
  \coordinate (v) at (0,.05);
  \coordinate (w) at (1.65,-.9);
  \foreach \p/\q/\side in {x/u/above left,x/v/right,x/w/above right,u/v/above left,v/w/above right,u/w/below}
    \draw[edge] (\p) -- node[sign,\side] {\(\pm\)} (\q);
  \draw[edge] (0,2.2) circle[radius=.25cm];
  \foreach \p in {x,u,v,w} \fill (\p) circle[radius=1.65pt];
  \node[right,inner sep=6pt,yshift=-1pt] at (x) {\(x\)};
  \node[left] at (u) {\(u\)};
  \node[below] at (v) {\(v\)};
  \node[right] at (w) {\(w\)};
  \node at (0,-1.95) {\(N_{\triangle}\)};
\end{scope}
\foreach \j/\dx/\dy in {0/7/0,1/0/-5.15,2/7/-5.15}{
\begin{scope}[xshift=\dx cm,yshift=\dy cm]
  \coordinate (x) at (0,1.95);
  \coordinate (a) at (-2.1,0);
  \coordinate (b) at (-.7,0);
  \coordinate (c) at (.7,0);
  \coordinate (d) at (2.1,0);
  \foreach \p/\q/\side in {x/a/above left,x/b/above left,x/c/above right,x/d/above right,a/b/above,b/c/above,c/d/above}
    \draw[edge] (\p) -- node[sign,\side] {\(\pm\)} (\q);
  \draw[edge] (a) .. controls (-1.8,-.95) and (.4,-.95) ..
    node[sign,below,pos=.5] {\(+\)} (c);
  \draw[edge,crossing] (b) .. controls (-.4,-.95) and (1.8,-.95) ..
    node[sign,below,pos=.5] {\(+\)} (d);
  \ifnum\j>0
    \draw[edge] (a) .. controls (-1.8,-1.8) and (1.8,-1.8) ..
      node[sign,below,pos=.5] {\(\ifnum\j=1 +\text{ or }-\else\pm\fi\)} (d);
  \fi
  \draw[edge] (0,2.2) circle[radius=.25cm];
  \foreach \p in {x,a,b,c,d} \fill (\p) circle[radius=1.65pt];
  \node[right,inner sep=6pt,yshift=-1pt] at (x) {\(x\)};
  \node[left] at (a) {\(a\)};
  \node[below left] at (b) {\(b\)};
  \node[below right] at (c) {\(c\)};
  \node[right] at (d) {\(d\)};
  \node at (0,-1.95) {\(\mathsf P_{\j}\)};
\end{scope}
}
\end{tikzpicture}
\caption{The signed graphs of \(N_{\triangle}\), \(\mathsf P_0\),
\(\mathsf P_1\), and \(\mathsf P_2\).}
\label{fig:finite-models}
\end{figure}

By the rank formula~\eqref{eq:frame-rank}, each of these arrangements
has the same matroid over every field of characteristic different from
\(2\). For a sufficiently large prime \(q\), the finite-field method
\cite{Athanasiadis1996} computes \(\chi(q)\) by counting points outside
the hyperplanes over \(\mathbb F_q\). Write
\((z)_{\underline\ell}=z(z-1)\cdots(z-\ell+1)\).
For \(\mathsf F_{s,\ell}\), choose the nonzero values at the vertices of
\(S\), no two equal up to sign, in \(\prod_{i=0}^{s-1}(q-2i-1)\) ways.
After excluding the values on \(S\) and their negatives,
\(Q=q-2s-1\) nonzero values remain.
If the value at \(c\) is zero, the values at \(z_1,\ldots,z_\ell\) must
be distinct and nonzero, giving \((Q)_{\underline\ell}\) choices.
For each of the \(Q\) nonzero values at \(c\), the \(z_i\) must avoid
that value and its negative but may be zero. They can therefore take
distinct values from a set of size \(Q-1\). Adding the two cases gives
\((Q)_{\underline\ell}+Q(Q-1)_{\underline\ell}
=(Q)_{\underline\ell}(Q-\ell+1)\), and hence
\[
 \chi_{\mathsf F_{s,\ell}}(t)
 =\left(\prod_{i=0}^{s-1}(t-2i-1)\right)
   (t-2s-1)_{\underline\ell}(t-2s-\ell).
\]

For \(N_{\triangle}\) and \(\mathsf P_0\), first fix a nonzero value
at \(x\). Each remaining vertex has \(Q=q-2\) available values,
including zero. In \(N_{\triangle}\), the values at \(u,v,w\)
are either all nonzero or exactly one is zero. These cases contribute
\((Q-1)(Q-3)(Q-5)\) and \(3(Q-1)(Q-3)\), respectively, giving
\((Q-1)(Q-3)(Q-2)\) in total. In \(\mathsf P_0\), at most one of
the values at \(b,c\) can be zero, since both signs occur on \(bc\).
If both values are nonzero, there are \((Q-1)(Q-3)\) choices. The value
at \(a\) must avoid the value at \(b\), its negative, and the value
at \(c\), leaving \(Q-3\) choices. Similarly, \(d\) has \(Q-3\)
choices, independent of the choice at \(a\). If exactly one of \(b,c\)
is zero, the \(2(Q-1)\) choices at these vertices each allow
\((Q-2)(Q-3)\) choices at \(a,d\). Thus the count for
\(\mathsf P_0\) is \((Q-1)(Q-3)\bigl((Q-3)^2+2(Q-2)\bigr)\).
In both models, multiplying by \(q-1\) accounts for the choice at \(x\).

\begin{samepage}
Adding either link on \(ad\) to \(\mathsf P_0\) gives \(\mathsf P_1\).
Restriction to the added hyperplane identifies the coordinates at \(a,d\)
up to sign, so the restriction is linearly isomorphic to \(N_{\triangle}\).
Adding the opposite link next gives \(\mathsf P_2\). On restriction to
this second hyperplane, the first \(ad\)-link becomes a coordinate
hyperplane at the identified vertex. This restriction is therefore
linearly isomorphic to \(\mathsf F_{2,1}\).
Deletion--restriction now determines the two
remaining polynomials. With \(\Delta(t)=(t-1)(t-3)(t-5)\), the results are
\[
\begin{aligned}
 \chi_{N_{\triangle}}(t)&=(t-1)(t-3)(t-4)(t-5),\\
 \chi_{\mathsf P_0}(t)&=\Delta(t)(t^2-8t+17),\\
 \chi_{\mathsf P_1}(t)&=\Delta(t)(t^2-9t+21),\\
 \chi_{\mathsf P_2}(t)&=\Delta(t)(t^2-10t+26).
\end{aligned}
\]

\end{samepage}

\begin{corollary}
\label{cor:finite-model-obstructions}
For the arrangements defined above, the following statements hold.
\begin{enumerate}[label=\textup{(\roman*)}]
 \item \(N_{\triangle}\) admits no nice partition.
 \item None of \(\mathsf P_0,\mathsf P_1,\mathsf P_2\) admits a nice partition.
 \item For \((s,\ell)=(1,2),(2,1),(2,2)\), every nice partition of
 \(\mathsf F_{s,\ell}\) has at most one part of size three.
\end{enumerate}
\end{corollary}

\begin{proof}
For \textup{(i)}, a nice partition of \(N_{\triangle}\) would have
part sizes \(1,3,4,5\), by
Proposition~\ref{prop:nice-localizes}\textup{(iv)}.
Let \(B_4,B_5\) be the parts of sizes four and five. By
Lemma~\ref{lem:global-support-trichotomy}, both are augmented stars.
Their centres lie in \(\{x,u,v,w\}\), since \(h_x\) is the only
element incident with the auxiliary vertex \(0\).
Opposite links on a pair in \(\{u,v,w\}\) form a two-point line and
cannot belong to the same part. A part centred at \(u,v\), or \(w\)
therefore has at most four elements, so \(B_5\) is centred at \(x\).
If \(h_x\in B_5\), at most one link on each of \(xu,xv,xw\) can
belong to \(B_5\), since a part cannot contain a line. This would give
\(|B_5|\le4\). Hence \(B_5\) consists of five of the six links incident with \(x\).
This leaves at most two elements for \(B_4\) if it is centred at
\(x\), and at most three if it is centred at \(u,v\), or \(w\)
(at most two links within \(\{u,v,w\}\) and one to \(x\)).
Both cases contradict \(|B_4|=4\).

For \textup{(ii)}, the quadratic factors have discriminants \(-4,-3,-4\),
respectively, so the characteristic polynomials do not split over the
integers, contrary to Proposition~\ref{prop:nice-localizes}\textup{(iv)}.
For \textup{(iii)}, the characteristic polynomial of each of
\(\mathsf F_{1,2}\), \(\mathsf F_{2,1}\), and \(\mathsf F_{2,2}\)
has \(3\) as a simple root. Two parts of size three would give a
factor \((t-3)^2\), by the same proposition.
\end{proof}

\begin{proposition}
\label{prop:attached-double-component-structure}
Let \(C\) be the vertex set of a component of \(D_U\) with \(|C|\ge2\)
and \(S_C\ne\varnothing\). Then \(D_U[C]\) is a star, and every
two distinct leaves are joined by exactly one signed link in \(\Sigma\).
\end{proposition}

\begin{proof}
Choose \(x\in S_C\). If \(u,v,w\) formed a triangle in \(D_U[C]\),
the matroid of the localization on \(\{x,u,v,w\}\) would be
isomorphic to \(M(N_{\triangle})\), by
Lemma~\ref{lem:double-component-propagation}, contradicting
Corollary~\ref{cor:finite-model-obstructions}\textup{(i)}.
Thus \(D_U[C]\) is triangle-free.
Now let \(u-v-w\) be a path of length two in \(D_U[C]\).
Lemma~\ref{lem:double-component-propagation} gives both signs on
\(xu,xv,xw\), and triangle-freeness excludes a double edge on \(uw\).
If no link on \(uw\) were present, the matroid \(N\) of the localization
on \(\{x,u,v,w\}\) would have rank four and eleven elements. Its lines with
at least three points are the three lines supported on the pairs
\(xu,xv,xw\), four balanced triangles on \(\{x,u,v\}\), and four on
\(\{x,v,w\}\), so \(\lambda_2(N)=11\). Lemma~\ref{lem:line-excess-bound}
would require \(\lambda_2(N)\ge35/3\), a contradiction. Hence \(uw\)
carries exactly one link.

Suppose next that \(a-b-c-d\) is a path of length three in \(D_U[C]\).
Both signs occur between \(x\) and each of \(a,b,c,d\), and the preceding
argument gives exactly one signed link on each of \(ac,bd\).
Since these pairs have disjoint endpoints, switching makes both links
positive while preserving all double edges. The matroid of the
localization on \(\{x,a,b,c,d\}\) is therefore isomorphic to
\(M(\mathsf P_j)\), where \(j\) is the number of links on \(ad\),
contradicting
Corollary~\ref{cor:finite-model-obstructions}\textup{(ii)}. Thus
\(D_U[C]\) is connected, triangle-free, and has no path of length three,
so it is a star. If \(c\) is its centre, any two distinct vertices of
\(C\setminus\{c\}\) are joined by a path of length two, so the preceding
argument gives exactly one signed link between them.
\end{proof}
\subsection{Existence of safe elements}
\label{subsec:existence-safe-elements}

\begin{lemma}
\label{lem:triangle-part-coordinate-support}
The support of every triangle part of \(\pi\) meets \(L\) in at
most one vertex.
\end{lemma}

\begin{proof}
Let \(B=\{e_{ab},e_{ac},e_{bc}\}\) be a triangle part, where
\(e_{ij}\) denotes its link supported on \(ij\). All three opposite
links of \(B\) are present by Lemma~\ref{lem:global-support-trichotomy}.
Suppose that \(a,b\in L\). The line supported on \(ab\) meets \(B\)
only in \(e_{ab}\). Since it meets exactly two parts, the elements
\(h_a,h_b,\bar e_{ab}\) belong to one other part, say \(D\).
The line supported on \(bc\) likewise meets \(B\) only in \(e_{bc}\).
As \(h_b\in D\), its other elements, including \(\bar e_{bc}\), lie
in \(D\). But \(h_a\) and \(\bar e_{bc}\) have disjoint augmented
supports \(0a\) and \(bc\), so they form a two-point line contained
in \(D\), contradicting niceness.
\end{proof}

\begin{proposition}
\label{prop:triangle-coordinate-safe-element}
If \(\pi\) has a triangle part whose support meets \(L\), then
\(\A\) has a \(\pi\)-safe element.
\end{proposition}

\begin{proof}
Suppose that \(\A\) has no \(\pi\)-safe element, and let \(B\)
be a triangle part whose support meets \(L\). By
Lemma~\ref{lem:triangle-part-coordinate-support}, exactly one of its
vertices, say \(x\), lies in \(L\). Both signs occur on each side
of its triangle, by Lemma~\ref{lem:global-support-trichotomy}.
Its two vertices in \(U\) therefore lie in the vertex set \(C\)
of one component of \(D_U\), with
\(|C|\ge2\) and \(x\in S_C\). By
Proposition~\ref{prop:attached-double-component-structure},
\(D_U[C]\) is a star and every two distinct leaves are joined by
exactly one signed link in \(\Sigma\). Write \(B=\{q_0,q_1,q_2\}\), with supports
\(cz,cx,zx\), respectively, where \(c\) is the centre of the star
and \(z\) is a leaf.

We first show that \(B\) is the only triangle part supported on two
vertices of \(C\) and one vertex of \(S_C\).
Suppose that \(B'\ne B\) is another such part. It also uses the
centre and a leaf. The parts \(B\) and \(B'\) cannot use the same
three vertices, since their signs would be opposite on all three sides,
making \(B'\) balanced. Consider the localization on the union of
the vertex sets of \(B\) and \(B'\). If two leaves occur, switch
the link between them to positive.
By Lemma~\ref{lem:double-component-propagation}, the matroid of this
localization is isomorphic to \(M(\mathsf F_{s,\ell})\), where \(s\)
is the number of vertices in \(L\) and \(\ell\) the number of leaves
used by the two parts. The possible pairs \((s,\ell)\) are
\((1,2)\), \((2,1)\), and \((2,2)\).
In each case, the induced nice partition has two parts of size three,
contrary to
Corollary~\ref{cor:finite-model-obstructions}\textup{(iii)}.

The three-point lines \(\{h_x,q_1,\bar q_1\}\) and
\(\{h_x,q_2,\bar q_2\}\) give
\(T=\Pi(h_x)=\Pi(\bar q_1)=\Pi(\bar q_2)\ne B\).
By Lemma~\ref{lem:global-support-trichotomy}, \(T\) is an augmented
star. The augmented supports \(0x,cx,zx\) have only \(x\) in
common, so \(T\) is centred at \(x\).
Let \(p_0,\ldots,p_m\) be an anchored route at \(c\) starting from
\(p_0=q_0\), with entered parts \(B_1,\ldots,B_m\).
By Lemma~\ref{lem:terminal-bad-route-hits-real-D3}, the route ends
in a triangle part \(B_m\), and the entered parts are pairwise
distinct. We claim that \(B_m=B\). Its triangle contains
\(c\). Both signs occur on each of its sides, so all its vertices
in \(U\) lie in \(C\). It cannot lie entirely in \(U\), since
\(D_U[C]\) is a star, and
Lemma~\ref{lem:triangle-part-coordinate-support} excludes two vertices in
\(L\). Thus \(B_m\) is supported on two vertices of \(C\) and
one vertex of \(S_C\), so uniqueness gives \(B_m=B\).
Hence \(p_m\in\{q_0,q_1\}\), and
Lemma~\ref{lem:anchored-route-no-return}\textup{(i)} excludes
\(p_m=p_0=q_0\). Therefore \(p_m=q_1\).

We now show that the route must enter \(T\) twice, contradicting the distinctness of its entered parts. For each \(i\), let \(\alpha_i\) be the defining form of \(p_i\)
with coefficient \(1\) at \(x_c\). Since \(B\) is unbalanced,
\eqref{eq:signed-frame-three-circuits} gives
\(\ker(\alpha_m-\alpha_0)=\bar q_2\).
Lemma~\ref{lem:anchored-route-no-return}\textup{(ii)} gives
\(B_t=\Pi(\bar q_2)=T\) for some
\(1\le t\le m\). Since \(B_m=B\ne T\), we have \(t<m\).
The link \(p_t\) is incident with \(c\) and belongs to the part
\(T\) centred at \(x\), so it lies on \(cx\). As \(q_1\in B\),
we have \(p_t=\bar q_1\). The links \(p_t\) and \(p_m=q_1\)
are opposite, so \(\ker(\alpha_m-\alpha_t)=h_x\). Applying the
same lemma to the segment \(p_t,\ldots,p_m\) now gives
\(T=\Pi(h_x)\in\{B_{t+1},\ldots,B_m\}\), contradicting
\(B_t=T\) because the entered parts are pairwise distinct.
\end{proof}

\begin{lemma}
\label{lem:full-coordinate-safe-element}
If \(L=V\) and \(\A\ne\varnothing\), then \(\A\) has a
\(\pi\)-safe element.
\end{lemma}

\begin{proof}
We may assume that the underlying simple graph is connected and has at least
two vertices. Its components give direct summands with nice induced
partitions, and safety in a summand implies safety in \(M\).
The coordinate element at an isolated vertex is safe. The circuits
\(\{h_i,h_j,H_{ij}^{\sigma}\}\) along the edges of the connected
graph show that \(M\) is connected. Write \(n=|V|\). Since all
coordinate elements are present, \(\rk M=n\), and \(\pi\) has
\(n\) parts by Proposition~\ref{prop:nice-localizes}\textup{(i)}.
Exactly one part is a singleton, by the observation in
Subsection~\ref{subsec:nice-partitions}.

The coordinate cores cover \(V\), by
Lemma~\ref{lem:coordinate-core-cover}, and each has size at most two.
On every four-point line, the induced partition has two parts of
sizes \(1\) and \(3\). A nonsingleton part with core \(\{i,j\}\)
would consist of the two opposite links on \(ij\), giving a part of
size two on \(\{h_i,h_j,H^+_{ij},H^-_{ij}\}\). Thus a core of size
two occurs only for a part consisting of a single link.
If the singleton part consists of a coordinate element, the \(n\)
cores have size at most one and cover \(n\) vertices, so they are
distinct singletons. Otherwise, the singleton part is \(\{g\}\),
where \(g\) is a link on \(ab\). The \(n\) coordinate elements lie
in the other \(n-1\) parts, so some part \(T\) contains two of them
and has empty core. The covering property gives
\[
 n\leq\sum_{B\in\pi}|K(B)|\leq 2+(n-2)=n.
\]
Equality holds throughout. Hence \(T\) is the only empty-core part,
and the remaining cores, apart from \(K(\{g\})=\{a,b\}\), are
distinct singletons covering \(V\setminus\{a,b\}\).
Both \(h_a,h_b\) lie in \(T\), since neither can belong to
\(\{g\}\) or to a part with a different vertex in its core.
Every link in \(T\) must meet both \(a\) and \(b\), since otherwise
it would form a two-point line in \(T\) with \(h_a\) or \(h_b\).
Thus every such link is on \(ab\). In either case, let \(W\) consist
of the vertices \(v\) for which some part has core \(\{v\}\), and
write \(B_v\) for that unique part.
If \(W=\varnothing\), then \(V=\{a,b\}\) and the parts are
\(T\) and \(\{g\}\). Every element of \(T\) is safe, since the
only other part is a singleton, proving the claim in this case.

We may therefore assume \(W\ne\varnothing\). Every link joining two
vertices of \(W\) lies in one of their parts. Orient a link on
\(uv\) from \(u\) to \(v\) if it belongs to \(B_u\), and from
\(v\) to \(u\) otherwise. If there is a directed simple cycle,
including a possible digon, its edges belong to the distinct parts
\(B_v\) indexed by its vertices. They therefore form a partial
transversal of \(\pi\). A balanced cycle is a circuit, contrary to
independence. For an unbalanced cycle on \(m\) vertices, the normal
vectors span the coordinate subspace on those vertices. Its closure
\(F\) thus has rank \(m\) and contains \(h_v\) at every cycle
vertex \(v\). Since the cycle edges already meet \(m\) distinct
parts, Proposition~\ref{prop:nice-localizes}\textup{(iii)} shows
that these are all the parts meeting \(F\). In particular, the
empty-core part, when present, does not meet \(F\). Each \(h_v\)
must therefore belong to \(B_v\), since it cannot belong to
\(B_w\) for \(w\ne v\). Every induced part on \(F\) contains
a cycle edge and its coordinate element, contrary to the singleton
condition. Hence the orientation is acyclic.

Choose a source \(v\) in this orientation. Every link from \(v\)
to another vertex of \(W\) belongs to \(B_v\). In the case with
singleton part \(\{g\}\) on \(ab\), a link from \(v\) to \(a\)
or \(b\) lies neither in \(T\) nor in \(\{g\}\). Among the remaining
cores, only \(\{v\}\) is contained in its support, so it also
belongs to \(B_v\). Thus every link incident with \(v\) belongs
to \(B_v\). We claim that \(h_v\) is safe. In a circuit
\(\{h_v,h_u,H_{vu}^{\sigma}\}\), the last two elements belong to
different parts, since \(H_{vu}^{\sigma}\in B_v\) and
\(h_u\notin B_v\). In a circuit
\(\{h_v,H^+_{vu},H^-_{vu}\}\), both links belong to \(B_v\).
The induced part sizes \(1,3\) on the four-point line, together
with \(h_u\notin B_v\), force \(h_v\in B_v\), so this circuit
is not bad either. These are all three-point circuits through \(h_v\),
by~\eqref{eq:signed-frame-three-circuits}, proving the claim.
\end{proof}

\begin{lemma}
\label{lem:global-safe-element}
If \(\A\ne\varnothing\), then \(\A\) has a \(\pi\)-safe element.
\end{lemma}
\begin{proof}
Suppose that \(M\) has no \(\pi\)-safe element. We first show that
no link joins \(u\in U\) to \(x\in L\). If such a link existed,
an anchored route at \(u\) starting from it would end in a triangle
part \(B\) whose support contains \(u\), by
Lemma~\ref{lem:terminal-bad-route-hits-real-D3}.
Proposition~\ref{prop:triangle-coordinate-safe-element} forces all three
vertices of \(B\) to lie in \(U\). Both signs occur on each side
of its triangle, so these vertices lie in the vertex set \(C\)
of one component of \(D_U\). The original link on \(ux\) gives
\(x\in S_C\). Proposition~\ref{prop:attached-double-component-structure}
then makes \(D_U[C]\) a star, contradicting the triangle in it.
Consequently the elements supported on \(L\) form a direct summand
of \(M\), represented by an arrangement containing every coordinate
hyperplane on \(L\). If \(L\ne\varnothing\),
Lemma~\ref{lem:full-coordinate-safe-element} gives a safe element
in this summand, hence in \(M\). If \(L=\varnothing\),
Lemma~\ref{lem:loopless-safe-element} gives a safe element directly.
Both cases contradict our assumption.
\end{proof}

\subsection{Proof of the main theorem}
\label{sec:main-proof}

\begin{proof}[Proof of Theorem~\ref{thm:main-intro}]
We induct on \(|\A|\), starting with the empty arrangement and its
empty partition. If \(\A\ne\varnothing\), choose a \(\pi\)-safe
element \(e\) by Lemma~\ref{lem:global-safe-element}.
Proposition~\ref{prop:three-circuit-deletion-bridge} gives a bijective
restriction map and nice induced partitions on the deletion and
restriction. Both arrangements are again signed-graphic, up to a
linear isomorphism. This is immediate for deletion. For restriction
to \(h_i\), substitute \(x_i=0\). For restriction to
\(H_{ij}^{\sigma}\), substitute \(x_i=\sigma x_j\).
After rescaling the equations and removing repetitions, the resulting
hyperplanes have equations \(x_k=0\) or \(x_k=\pm x_\ell\). Both
arrangements have fewer hyperplanes than \(\A\), so induction applies
to their induced partitions. The recursive step at \(e\)
shows that \((\A,\pi)\) is inductively factored.
\end{proof}

\section{Freeness and Terao's conjecture}
\label{sec:realization-space}

A nice partition gives a factorization of the characteristic
polynomial into linear factors, with the part sizes as its nonzero
roots \cite{Terao1992}. Such factorizations also occur for free
arrangements, whose characteristic polynomials have their exponents
as roots by Terao's factorization theorem
\cite[Theorem~4.137]{OrlikTerao1992}. Nevertheless, an arrangement
admitting a nice partition need not be free \cite{HogeRoehrle2016}.
Jambu and Paris~\cite{JambuParis1995} proved freeness under the
stronger assumption of an inductive factorization. For signed-graphic
arrangements, Theorem~\ref{thm:main-intro} shows that every nice
partition satisfies this stronger condition. Thus every signed-graphic
arrangement admitting a nice partition is free, with nonzero
exponents equal to the part sizes. The converse fails even for real
signed-graphic arrangements.

\begin{example}
\label{ex:free-not-nice}
Consider the real arrangement \(N_{\triangle}\) from
Subsection~\ref{subsec:local-coordinate-obstructions}.
Deleting \(h_x\) gives the reflection arrangement of type \(D_4\),
while restriction to \(h_x\) gives that of type \(B_3\).
These arrangements are free with exponents \(\{1,3,3,5\}\) and
\(\{1,3,5\}\), respectively
\cite[Propositions~6.82 and~6.85]{OrlikTerao1992}.
Terao's addition--deletion theorem
\cite[Theorem~4.51]{OrlikTerao1992} therefore shows that
\(N_{\triangle}\) is free with exponents \(\{1,3,4,5\}\).
However, Corollary~\ref{cor:finite-model-obstructions}\textup{(i)}
shows that it admits no nice partition.
\end{example}

The inductive factorization obtained in Theorem~\ref{thm:main-intro}
can be transferred to any arrangement with an isomorphic intersection
lattice. This relates our result to Terao's conjecture, which asserts
that freeness is determined
by the intersection lattice over a fixed field
\cite[Conjecture~4.138]{OrlikTerao1992}.
Let \(\A\) be a signed-graphic arrangement with a nice partition
\(\pi\), and let \(\mathcal B\) be a central arrangement with
\(L(\mathcal B)\cong L(\A)\). Fix an isomorphism between these
lattices. It induces a bijection between the hyperplanes and carries
\(\pi\) to a partition of \(\mathcal B\). At each step of an
inductive factorization, the corresponding deletions and restrictions
have isomorphic intersection lattices. The induced bijections identify
their partitions and commute with the restriction maps. Thus the
inductive factorization of \(\pi\) supplied by
Theorem~\ref{thm:main-intro} transfers to \(\mathcal B\).
Together with the exponent formula above, this gives the following
corollary.

\begin{corollary}
\label{cor:realization-transfer}
\label{cor:terao-signed-frame-nice}
Let \(\A\) be a signed-graphic arrangement of rank \(r\) with a nice
partition \(\pi=(\pi_1,\ldots,\pi_r)\), and let \(\mathcal B\) be a
central arrangement over \(\mathbb K\) with \(L(\mathcal B)\cong L(\A)\). Every
isomorphism \(L(\A)\cong L(\mathcal B)\) carries \(\pi\) to an
inductive factorization of \(\mathcal B\). In particular,
\(\mathcal B\) is free with nonzero exponents
\(|\pi_1|,\ldots,|\pi_r|\).
\end{corollary}

Terao's conjecture therefore holds for the intersection lattices of
signed-graphic arrangements admitting a nice partition.
Corollary~\ref{cor:realization-transfer} also applies to arrangements
that are not projectively equivalent to any signed-graphic arrangement,
as the following example shows. Here two arrangements are called
\emph{projectively equivalent} if an invertible linear change of
coordinates carries one arrangement onto the other.

\begin{example}
\label{ex:nonfano-realization-family}
Let \(\A_0\) be the arrangement in \(\mathbb R^4\) defined by
\(x_1(x_1^2-x_2^2)(x_1^2-x_3^2)(x_2^2-x_3^2)\), and set
\[
 \mathcal B=\A_0\cup
 \{\ker x_4,\ker(x_1-x_4),\ker(x_1-\sqrt2\,x_4)\}.
\]
Let \(\mathcal C\) be obtained from \(\mathcal B\) by replacing
\(\ker(x_1-\sqrt2\,x_4)\) with \(\ker(x_1+x_4)\).
Then \(\mathcal C\) is signed-graphic, and this replacement preserves
the intersection lattice. Indeed, the seven defining forms of
\(\A_0\) do not involve \(x_4\). In either arrangement, adjoining
one of the three new forms to any subset of the original forms
increases its rank by one. Any two new forms span
\(\langle x_1,x_4\rangle\), so adjoining two or three gives a rank
equal to one plus the dimension of the span of \(x_1\) and the chosen
original forms. Thus corresponding subsets have the same rank,
and \(L(\mathcal B)\cong L(\mathcal C)\).

The seven equations defining \(\A_0\), viewed in \(\mathbb R^3\),
define the non-Fano arrangement. Hoge and
R\"ohrle~\cite[Example~3.19]{HogeRoehrle2016} give the nice partition
\[
 \pi_0=\bigl(\{h_1\},
 \{H_{12}^{+},H_{12}^{-},H_{23}^{+}\},
 \{H_{13}^{+},H_{13}^{-},H_{23}^{-}\}\bigr).
\]
The same partition is nice for \(\A_0\), since the unused fourth
coordinate does not change the intersection lattice. Adding
\(\{h_4,H_{14}^{+},H_{14}^{-}\}\) as a fourth part gives an independent
partition of \(\mathcal C\), since each new defining form has a nonzero
\(x_4\)-coefficient. Any flat meeting the new part in at
least two elements contains \(h_1\), so the singleton condition is
also preserved.
Thus \(\mathcal C\) has a nice
partition with part sizes \(1,3,3,3\). Since
\(L(\mathcal B)\cong L(\mathcal C)\),
Corollary~\ref{cor:realization-transfer} shows that \(\mathcal B\)
is free with exponents \(\{1,3,3,3\}\).

Nevertheless, \(\mathcal B\) is not projectively equivalent to any
signed-graphic arrangement. Suppose otherwise. After a suitable real
linear change of coordinates, all its hyperplanes have defining
equations with rational coefficients. By a further change of coordinates
with rational coefficients, we may arrange that the images of
\(\ker x_1\) and \(\ker x_4\) are the coordinate hyperplanes
\(y_1=0\) and \(y_4=0\). The images of \(\ker(x_1-x_4)\) and
\(\ker(x_1-\sqrt2\,x_4)\) then have equations
\[
 y_1-t y_4=0,\qquad y_1-\sqrt2\,t y_4=0
\]
for some nonzero real number \(t\). Both hyperplanes must still admit
defining equations with rational coefficients. The first therefore
forces \(t\in\mathbb Q\), while the second forces
\(\sqrt2\,t\in\mathbb Q\), contradicting the irrationality of
\(\sqrt2\).
\end{example}

In the real case, the part sizes that determine the nonzero exponents
also occur in a formula counting regions. For a real arrangement
\(\A\) in \(W\), removing all its hyperplanes leaves the complement
\(W\setminus\bigcup_{H\in\A}H\), whose connected components are
called \emph{regions}. Let \(\mathcal R(\A)\) denote their set.
Fix a base region \(B\), and write \(S(B,R)\) for the set of
hyperplanes separating \(B\) from a region \(R\). Ordering the
regions by \(R\le R'\) if \(S(B,R)\subseteq S(B,R')\) gives the
\emph{poset of regions} \(P(\A,B)\), with rank function
\(R\mapsto|S(B,R)|\). Its rank-generating function therefore counts
regions according to the number of hyperplanes separating them from \(B\).
Jambu and Paris~\cite[Proposition~3.4 and Theorems~5.1,~6.1]{JambuParis1995}
proved that, for an inductive factorization
\(\pi=(\pi_1,\ldots,\pi_r)\), one can choose \(B\) so that
\(P(\A,B)\) is a lattice and its rank-generating function is a
product of factors \(1+t+\cdots+t^{|\pi_i|}\), one for each part.
Their results also apply when the common intersection of the
hyperplanes is nonzero, since passing to the quotient by this
intersection preserves both the inductive factorization and the
poset of regions. They further asked whether these conclusions hold
for real arrangements admitting a nice partition. For real
signed-graphic arrangements, Theorem~\ref{thm:main-intro} gives the
following affirmative answer.

\begin{corollary}\label{cor:regions-factorization}
Let \(\A\) be a real signed-graphic arrangement with a nice partition
\(\pi=(\pi_1,\ldots,\pi_r)\). There is a base region \(B\) such that
\(P(\A,B)\) is a lattice and its rank-generating function satisfies
\[
 \sum_{R\in\mathcal R(\A)}t^{|S(B,R)|}
 =\prod_{i=1}^r(1+t+\cdots+t^{|\pi_i|}).
\]
\end{corollary}

\begin{samepage}
\section*{Declaration on the use of AI}

During the preparation of this manuscript, the authors used AI-assisted
tools for language editing, \LaTeX{} formatting, consistency checks,
computational checks, and proof simplification. In particular, these
tools assisted with checking arguments involving local signed graphs, verifying
characteristic-polynomial calculations, and simplifying
the anchored-route arguments.
The authors independently verified the mathematical arguments and
reviewed, revised, and finalized all material prepared with AI assistance.
The authors take full responsibility for the correctness and
content of the manuscript.

\end{samepage}

\bigskip
\begin{samepage}
\small
\noindent
\textsc{College of Science, National University of Defense Technology,}\\
\textsc{Changsha 410073, Hunan, P. R. China}\\
\href{mailto:liangweikang08@163.com}{\texttt{liangweikang08@163.com}}
(Weikang Liang)\\
\href{mailto:yue.zhou.ovgu@gmail.com}{\texttt{yue.zhou.ovgu@gmail.com}}
(Yue Zhou)
\end{samepage}

\medskip
\begin{samepage}
\small
\noindent
\textsc{School of Mathematics, Hunan University,}\\
\textsc{Changsha 410082, Hunan, P. R. China}\\
\href{mailto:wangsuijie@hnu.edu.cn}{\texttt{wangsuijie@hnu.edu.cn}}
(Suijie Wang)
\end{samepage}

\end{document}